\documentclass[10pt,a4paper]{scrartcl}%

\usepackage[T1]{fontenc}
\usepackage{epsfig}
\usepackage{graphicx}
\usepackage{amssymb}
\usepackage{amsmath}
\usepackage[utf8]{inputenc}
\usepackage{amsfonts}
\usepackage{amsthm}
\usepackage{dsfont}
\usepackage{comment}
\usepackage{enumitem}
\usepackage{mathbbol}
\usepackage{mathrsfs}

\usepackage[a4paper,vmargin={2cm,2cm},hmargin={2.25cm,2.25cm}]{geometry}

\usepackage{hyperref}
\usepackage[english]{babel}

\usepackage{listings}

\usepackage{pgf,tikz,pgfplots}

\pgfplotsset{compat=1.15}
\usetikzlibrary{arrows,shapes,positioning,calc}

\theoremstyle{plain}
\newtheorem{thm}{Theorem}

\newtheorem{prop}[thm]{Proposition}
\newtheorem{cor}[thm]{Corollary}

\theoremstyle{definition}
\newtheorem{defi}[thm]{Definition}

\theoremstyle{remark}

\makeindex

\date{}

\begin{document}

\title{Degrees in random uniform minimal factorizations}
\author{Etienne Bellin \footnote{etienne.bellin@polytechnique.edu}
\\
{\normalsize CMAP - Ecole Polytechnique}

}

\maketitle

\begin{abstract}
    We are interested in random uniform minimal factorizations of the $n$-cycle which are factorizations of $(1~2\dots n)$ into a product of $n-1$ transpositions. Our main result is an explicit formula for the joint probability that 1 and 2 appear a given number of times in a uniform minimal factorization. For this purpose, we combine bijections with Cayley trees together with explicit computations of multivariate generating functions.
\end{abstract}

\section{Introduction}

Consider the cycle permutation $(1~2\dots n)$ for $n\geq 2$. A \textit{minimal factorization} is a $(n-1)$-tuple of transpositions $(\tau_1,\dots,\tau_{n-1})$ such that $\tau_{n-1} \circ \dots \circ \tau_1 = (1~2\dots n)$. We denote by $\mathfrak{M}_n$ the set of all minimal factorizations of the cycle $(1~2\dots n)$. Dénes \cite{denes} first showed that $\mathfrak{M}_n$ has cardinality $n^{n-2}$ and several bijective proofs followed afterwards (see \cite{moszk}, \cite{pepper}, \cite{goulden} and \cite{biane2005}). Minimal factorizations are linked to other combinatorial objects such as non-crossing partitions \cite{biane1997} and parking functions \cite{biane2002} and more general factorizations have deep connections with enumerative geometry (see e.g \cite{Alexandrov}).

Let $(\tau_1^{(n)},\dots,\tau_{n-1}^{(n)})$ be a random minimal factorization chosen uniformly at random in $\mathfrak{M}_n$. The study of the behaviour of such a randomly picked minimal factorization is recent (see \cite{kor2018}, \cite{kor2019} and \cite{thevenin}) and has a rich probabilistic structure: for instance, it is shown in \cite{thevenin} that such minimal factorizations have connections with Aldous-Pitman fragmentation of the Brownian continuum random tree. Here we are interested in the law of the number of times 1 and 2 appear in $(\tau_1^{(n)},\dots,\tau_{n-1}^{(n)})$. In \cite[Corollary 1.2 (iv)]{kor2019} it was obtained that:
$$
    \mathbb{P} \left( \mathbb{T}_1^{(n)}=i, \mathbb{T}_2^{(n)}=j \right) \xrightarrow[n \rightarrow \infty]{} e^{-2} \left[  \frac{i+j-2}{(i+j-1)!} +\frac{i+j-1}{i!j!} -\frac{i+j-1}{(i+j)!} \right].
$$
where $\mathbb{T}_k^{(n)} = \#\{ 1 \leq \ell \leq n-1 ~: \tau_\ell^{(n)}(k) \neq k \}$ is the number of time $k$ appears in a transposition. We refine this result by finding explicitly the joint distribution for fixed $n$:

\begin{thm}
For $i,j \geq 1$ and $n \geq i+j$:
\begin{equation}
    \mathbb{P} \left( \mathbb{T}_1^{(n)}=i, \mathbb{T}_2^{(n)}=j \right) = \frac{n!(n-1)^{n-i-j-1}}{(n-i-j)!(n+1)^{n-1}} \left[  \frac{(i+j-2)(n-1)}{(i+j-1)!(n-i-j+1)} +\frac{i+j-1}{i!j!} -\frac{i+j-1}{(i+j)!} \right].
    \label{main formula}
\end{equation}
\label{main thm}
\end{thm}

\medskip

To show Theorem \ref{main thm}, we explicitly compute the exponential generating function of the (normalized) trivariate generating function $G_n$ defined by
$$
G_n(x,y,z) = n^{n-2}\mathbb{E} \left[x^{\mathbb{T}_1^{(n)}}  y^{\mathbb{T}_2^{(n)}} z^{\mathbb{M}_1^{(n)}} \right]
$$
where $\mathbb{M}_k^{(n)}= \# \{ 1 \leq \ell \leq n-1 ~: \tau_1^{(n)}(k) \dots \tau_{\ell-1}^{(n)}(k) \neq \tau_1^{(n)}(k) \dots \tau_{\ell}^{(n)}(k) \}$ is the number of transpositions that affect the trajectory of $k$, and then extract the coefficient $[x^iy^j]G_n(x, y, 1)$. To this end, there are 4 main steps. First, using a known bijection between minimal factorizations and Cayley trees, we reformulate the problem in terms of a generating function of a trivariate statistic on Cayley trees (Section \ref{bij min fact and cayley trees}). To compute this generating function, we actually start by computing another generating function $F_n$ obtained by changing one of the three statistics (Section \ref{computation of fn}). This also yields a result of independent
interest by confirming a conjecture \cite{caraceni} involving distributional symmetries in uniform Cayley trees (Corollary \ref{coro}). Finally, we show bijectively
that $G_n(x, y, z) = G_n(y, x, z)$ (Section \ref{a sym result}), and by combining this with the explicit formula of $F_n$ we get the exponential generating function of $G_n$ (Section \ref{comp EGF of gn}) and Theorem \ref{main thm} follows (Section \ref{proof main thm}).

\section{Bijection between minimal factorizations and Cayley trees} \label{bij min fact and cayley trees}

\subsection{The bijection} \label{the bij}

Here we explain how to associate a labeled tree with a minimal factorization which will be an essential tool for us. We refer to \cite{kor2019} for details and proofs. Fix an integer $n \geq 2$.

\begin{defi}
For $(\tau_1,\dots,\tau_{n-1})\in \mathfrak{M}_n$, we define a labeled tree $\mathcal{F} \left( \tau_1,\dots,\tau_{n-1} \right)$ with $n$ vertices labeled from 1 to $n$ where an edge labeled $l$ is drawn between the vertices labeled $a$ and $b$ if and only if $\tau_l = (a,b)$ (see figure \ref{ex of f} for an example).
\end{defi}

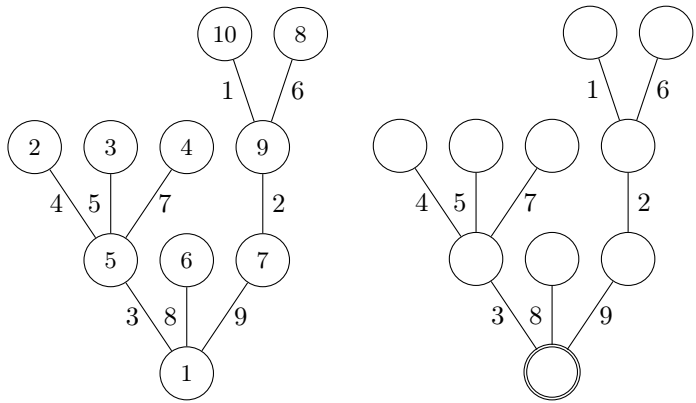
\begin{figure}[!h]
\begin{center}
    \begin{tikzpicture}
    [root/.style = {draw,circle,double,minimum size = 20pt,font=\small},
    vertex/.style = {draw,circle,minimum size = 20pt, font=\small}]
        \node[vertex] (1) at (0,0) {1};
        \node[vertex] (2) at (-1,1.5) {5};
        \node[vertex] (3) at (0,1.5) {6};
        \node[vertex] (4) at (1,1.5) {7};
        \node[vertex] (5) at (-2,3) {2};
        \node[vertex] (6) at (-1,3) {3};
        \node[vertex] (7) at (0,3) {4};
        \node[vertex] (8) at (1,3) {9};
        \node[vertex] (9) at (0.5,4.5) {10};
        \node[vertex] (10) at (1.5,4.5) {8};
        \draw (1) -- node[left] {3} (2) ;
        \draw (1) -- node[left] {8} (3);
        \draw (1) -- node[right] {9} (4);
        \draw (2) -- node[left] {4} (5);
        \draw (2) -- node[left] {5} (6);
        \draw (2) -- node[right] {7} (7);
        \draw (4) -- node[right] {2} (8);
        \draw (8) -- node[left] {1} (9);
        \draw (8) -- node[right] {6} (10);
    \end{tikzpicture}
    \hspace{1em}
    \begin{tikzpicture}
    [root/.style = {draw,circle,double,minimum size = 20pt,font=\small},
    vertex/.style = {draw,circle,minimum size = 20pt, font=\small}]
        \node[root] (1) at (0,0) {};
        \node[vertex] (2) at (-1,1.5) {};
        \node[vertex] (3) at (0,1.5) {};
        \node[vertex] (4) at (1,1.5) {};
        \node[vertex] (5) at (-2,3) {};
        \node[vertex] (6) at (-1,3) {};
        \node[vertex] (7) at (0,3) {};
        \node[vertex] (8) at (1,3) {};
        \node[vertex] (9) at (0.5,4.5) {};
        \node[vertex] (10) at (1.5,4.5) {};
        \draw (1) -- node[left] {3} (2) ;
        \draw (1) -- node[left] {8} (3);
        \draw (1) -- node[right] {9} (4);
        \draw (2) -- node[left] {4} (5);
        \draw (2) -- node[left] {5} (6);
        \draw (2) -- node[right] {7} (7);
        \draw (4) -- node[right] {2} (8);
        \draw (8) -- node[left] {1} (9);
        \draw (8) -- node[right] {6} (10);
    \end{tikzpicture}
\caption{Representation of $\mathcal{F}$ on the left and $\mathcal{E}$ on the right when $n=10$ for the minimal factorization $((9~10),(7~9),(1~5),(2~5),(3~5),(8~9),(4~5),(1~6),(1~7))$ of $(1\dots 10)$. The double circle represents the root of the tree $\mathcal{E}$.} 
\label{ex of f}
\end{center}
\end{figure}

Clearly $\mathcal{F}$ is injective since a minimal factorization can be easily read on its associated tree. Actually the tree $\mathcal{F}$ gives too much information, indeed it is still possible to retrieve the associated minimal factorization when we erase the vertex-labels and keep only the edge-labels. More precisely we have: 

\begin{defi}
If $(\tau_1,\dots,\tau_{n-1})\in \mathfrak{M}_n$, then we construct a rooted, edge-labeled tree $\mathcal{E}(\tau_1,\dots,\tau_{n-1})$ by doing the following on the tree $\mathcal{F}(\tau_1,\dots,\tau_{n-1})$:
\begin{itemize}
    \item We root the tree at the vertex labeled 1.
    \item We erase all the vertex-labels (and keep only the edge-labels).
\end{itemize}
\end{defi}

\begin{prop}
The map $\mathcal{E}$ gives a bijection between the set $\mathfrak{M}_n$ and the set $\mathfrak{C}'_n$ of rooted trees with $n-1$ edges labeled from 1 to $n-1$.
\label{e is a bij}
\end{prop}

The set $\mathfrak{C}'_n$ is clearly in bijection with the set $\mathfrak{C}_n$ of Cayley trees with $n$ vertices (\textit{i.e.} trees with $n$ vertices labeled from 1 to $n$). Indeed if $t \in \mathfrak{C}_n$ we create $\alpha(t) \in \mathfrak{C}'_n$ by rooting the tree $t$ at the vertex labeled 1, then by pulling all the vertex-labels (except 1 which is erased from $\alpha(t)$) towards the root into the nearest edge. We then subtract 1 from all the labels (see figure \ref{ex of t to t'} for an example). The map $\alpha$ is clearly a bijection.

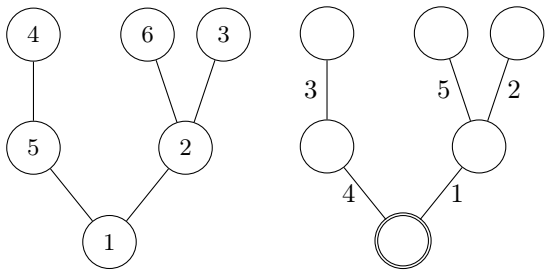
\begin{figure}[!h]
\begin{center}
    \begin{tikzpicture}
    [root/.style = {draw,circle,double,minimum size = 20pt,font=\small},
    vertex/.style = {draw,circle,minimum size = 20pt, font=\small}]
        \node[vertex] (1) at (0,0) {1};
        \node[vertex] (2) at (-1,1.25) {5};
        \node[vertex] (3) at (1,1.25) {2};
        \node[vertex] (4) at (-1,2.75) {4};
        \node[vertex] (5) at (0.5,2.75) {6};
        \node[vertex] (6) at (1.5,2.75) {3};

        \draw (1) -- (2) ;
        \draw (1) -- (3);
        \draw (2) -- (4);
        \draw (3) -- (5);
        \draw (3) -- (6);
    \end{tikzpicture}
    \hspace{1 em}
    \begin{tikzpicture}
    [root/.style = {draw,circle,double,minimum size = 20pt,font=\small},
    vertex/.style = {draw,circle,minimum size = 20pt, font=\small}]
        \node[root] (1) at (0,0) {};
        \node[vertex] (2) at (-1,1.25) {};
        \node[vertex] (3) at (1,1.25) {};
        \node[vertex] (4) at (-1,2.75) {};
        \node[vertex] (5) at (0.5,2.75) {};
        \node[vertex] (6) at (1.5,2.75) {};

        \draw (1) -- node[left] {4} (2) ;
        \draw (1) -- node[right] {1} (3);
        \draw (2) -- node[left] {3} (4);
        \draw (3) -- node[left] {5} (5);
        \draw (3) -- node[right] {2} (6);
    \end{tikzpicture}    
\caption{A tree $t \in \mathfrak{C}_6$ on the left transformed into $\alpha(t) \in \mathfrak{C}'_6$ on the right by pulling the vertex-labels towards the root and subtracting 1.} 
\label{ex of t to t'}
\end{center}
\end{figure}

\medskip

In particular $\mathfrak{M}_n$ has the same cardinality as the set of Cayley trees with $n$ vertices $\mathfrak{C}_n$ which is known to be $n^{n-2}$. This explains the renormalizing term in the definition of $G_n$. In the article \cite{kor2019} the authors give an explicit algorithm to find a minimal factorization back from its associated tree $\mathcal{E}$. We explain this algorithm right now because it will be essential for what follows. 

\medskip

Let $t \in  \mathfrak{C}'_n$ and $f \in \mathfrak{M}_n$ such that $t = \mathcal{E}(f)$. To recover $f$ from $t$, we will recover $\mathcal{F}(f)$ (from which it is immediate to find $f$ back). For this purpose, we will gradually assign the vertex-labels $1, 2, \dots, n$ to $t$ as they can be found in $\mathcal{F}(f)$. More precisely, let us describe an algorithm \verb&Next& that takes the tree $t$ and an integer $k \in \{1,\dots,n-1\}$ as arguments and assigns the vertex-label $k+1$ to $t$ if the vertex-label $k$ has already been assigned in $t$ (if $k$ is not assigned yet \verb&Next&(t,k) does nothing). The algorithm starts from the vertex $v_0$ of $t$ labeled $k$. Then it follows the longest possible path of edges $e_1,\dots,e_\ell$ in $t$ such that:
\begin{itemize}
    \item $e_1,\dots,e_\ell$ is a \textit{path of edges} meaning that for all $0<i<\ell$, $e_i$ and $e_{i+1}$ share a common vertex $v_i$.
    \item $e_1$ is the edge with smallest label adjacent to $v_0$.
    \item For all $0<i<\ell$, $e_{i+1}$ has the smallest label among the edges adjacent to $v_i$ having a label greater than $e_i$'s label.
\end{itemize}

Starting from $v_0$ and following the path $e_1,\dots,e_\ell$ leads to the vertex of $t$ which then gets the label $k+1$. We also initiate the algorithm with $\verb&Next&(t,0)$ which gives label 1 to the root of $t$.

\medskip

Now by applying successively $\verb&Next&$ to $t$ we can label all the vertices of $t$ in order to find back $\mathcal{F}(f)$. More precisely we have the following definition and Proposition:

\begin{defi}
For $1 \leq k \leq n$ and $t \in \mathfrak{C}'_n$, we denote by $\verb&Find&_k(t)$ the tree obtained by applying successively $\verb&Next&(t,0)$ then $\verb&Next&(t,1)$, $\verb&Next&(t,2)$, $\dots$, $\verb&Next&(t,k-1)$ to $t$.
\end{defi}

For example, applying $\verb&Find&_{10}$ to the tree on the right of figure \ref{ex of f} gives back the vertex-labels on the left of figure \ref{ex of f}.

\begin{prop}
For $t = \mathcal{E}(f) \in \mathfrak{C}'_n$:
$$
\verb&Find&_n(t) = \mathcal{F}(f).
$$
\end{prop}

\subsection{Reformulation in terms of Cayley trees}

In order to prove Theorem \ref{main thm}, we start with rewriting the generating function $G_n$ defined in the Introduction in terms of Cayley trees. To this purpose we introduce some notation. For $A$ a subset of $\{1,\dots,n\}$ we denote by $\mathfrak{C}_A$ the set of trees with $|A|$ vertices labeled in a one-to-one manner with the elements of $A$. Notice that if $A = \{1,\dots,n\}$ then $\mathfrak{C}_A=\mathfrak{C}_n$ is the set of \textit{Cayley trees} with $n$ vertices. If $t \in \mathfrak{C}_A$ and $i \in A$, we denote by $\deg_i(t)$ the degree of the vertex $i$ (where "the vertex $i$" has to be understood as "the vertex labeled $i$") in the tree $t$. Also, similarly to the algorithm \verb&Next& (introduced in the previous Section) we consider a certain path in $t$ starting from $i$, namely:

\begin{defi}
For $t \in \mathfrak{C}_A$ and $i \in A$ we consider the longest path of vertices, starting from $i$, such that each vertex of the path has the smallest label among the ones that are both adjacent to the previous vertex on the path and have a greater label than the label of the previous vertex on the path. We denote by $L_i(t)$ the set of vertices composing this path but without including the first one which is the vertex $i$ (see figure \ref{ex of Li(t)} for an example).
\end{defi}

\begin{figure}[!h]
\begin{center}
    \begin{tikzpicture}
    [root/.style = {draw,circle,double,minimum size = 20pt,font=\small},
    vertex/.style = {draw,circle,minimum size = 20pt, font=\small}]
        \node[vertex,red,thick] (1) at (0,0) {\textbf{3}};
        \node[vertex] (2) at (-1,1.5) {1};
        \node[vertex] (3) at (0,1.5) {2};
        \node[vertex,red,thick] (4) at (1,1.5) {\textbf{4}};
        \node[vertex] (5) at (-2,3) {11};
        \node[vertex] (6) at (-1,3) {12};
        \node[vertex,red,thick] (7) at (1,3) {\textbf{8}};
        \node[vertex] (8) at (2,3) {9};
        \node[vertex] (10) at (0.5,4.5) {5};
        \node[vertex] (11) at (1.5,4.5) {7};
        \draw[red,thick] (1) -- (2) ;
        \draw (1) -- (3);
        \draw[red,thick] (1) -- (4);
        \draw (2) -- (5);
        \draw (2) -- (6);
        \draw[red,thick] (4) -- (7);
        \draw (4) -- (8);
        \draw (7) -- (10);
        \draw (7) -- (11);
    \end{tikzpicture}
\caption{Example of a tree $t \in \mathfrak{C}_A$ with $A=\{1,2,3,4,5,7,8,9,11,12\}$. The path $L_1(t)=\{3,4,8\}$ starting from 1 is represented in thick red. Here $\deg'_2(t) = \deg_8(t) = 3$.} 
\label{ex of Li(t)}
\end{center}
\end{figure}
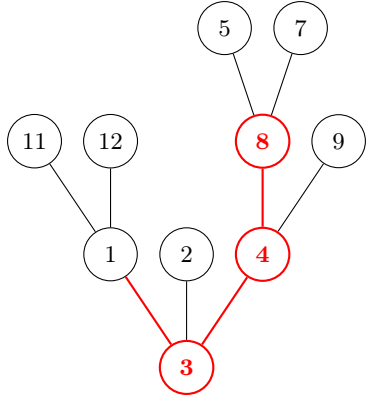

\begin{defi}
For $t \in \mathfrak{C}_n$ we denote by $\deg'_2(t)$ the degree, in $t$, of the last vertex of the path $L_1(t)$ (see figure \ref{ex of Li(t)} for an example).
\end{defi}

Recall from Section \ref{the bij} the bijection $\mathcal{E}$ between $\mathfrak{M}_n$ and $\mathfrak{C}'_n$ as well as a bijection $\alpha$ from $\mathfrak{C}_n$ to $\mathfrak{C}'_n$ (illustrated in figure \ref{ex of t to t'}). Notice that if $f \in \mathfrak{M}_n$ and $t=\alpha \circ \mathcal{E}(f)$ then
$$
(\mathbb{T}_1(f),\mathbb{T}_2(f),\mathbb{M}_1(f))=(\deg_1(t),\deg'_2(t), |L_1(t)|).
$$
where $\mathbb{T}_i(f)$ is the number of times $i$ appears in $f$ and $\mathbb{M}_j(f)$ is the number of transpositions in $f$ that affect the trajectory of $j$. The last identity allows us to reformulate the definition of $G_n$ for $n \geq 2$:
\begin{equation}
G_n(x,y,z) = \sum_{t \in \mathfrak{C}_n} x^{\deg_1(t)} y^{\deg'_2(t)} z^{|L_1(t)|}.
\label{gn with cayley trees}
\end{equation}
In the next Section, in order to compute $G_n$, we introduce another generating function $F_n$ whose definition is similar to $(\ref{gn with cayley trees})$ except that $\deg'_2$ is replaced with $\deg_2$.

\section{First generating function on Cayley trees} \label{computation of fn}

\noindent For $n\geq 2$ we define the generating function:
$$
F_n(x,y,z) = \sum_{t \in \mathfrak{C}_n} x^{\deg_1(t)} y^{\deg_2(t)} z^{|L_1(t)|}.
$$
By convention we also set for $n \geq 1$:
$$
F_n(x,1,1)=F_n(x)=\sum_{t \in \mathfrak{C}_n} x^{\deg_1(t)}=x(n-1+x)^{n-2}.
$$
The last equality is a well known result on Cayley trees. Actually, it turns out we have an explicit formula for $F_n(x,y,z)$.

\begin{prop}
For $n \geq 3$,
\begin{align}
    \nonumber
    F_n(x,y,z) = xyz & \left[ x\left( n-2+x \right)^{n-3} \left( 1 - \frac{yz}{z+y-1} \right) \right. \\
    &  \left. + \left( x+y+z+n-3 \right)^{n-3} \left( n-2+yz+\frac{xyz}{z+y-1} \right) \right].
    \label{formula fn}
\end{align}
\label{prop computation of fn}
\end{prop}

Formula (\ref{formula fn}) still holds in the cases $(y=1,n=2)$ and $(y=z=1,n=1)$. Proposition \ref{prop computation of fn} implies that $F_n$ is symmetric in $y$ and $z$ thus we have the following Corollary which confirms a conjecture made by Caraceni \cite{caraceni}:

\begin{cor}
Let $T_n$ be a random uniform Cayley tree with $n$ vertices, then $(\deg_1(T_n),|L_1(T_n)|,\deg_2(T_n))$ and $(\deg_1(T_n),\deg_2(T_n),|L_1(T_n)|)$ have the same law.
\label{coro}
\end{cor}

It would be very interesting to obtain a direct bijective proof of Corollary \ref{coro}. Formula (\ref{formula fn}) with $y=1$ was first conjectured in \cite[Conjecture 1.4]{kor2019} and was proved by O. Angel \& J. Martin \cite{angel}. Below, we give Angel and Martin's proof of the case $y=1$ which will be useful to deduce the general case. 

\begin{proof}[Proof of Proposition \ref{prop computation of fn} for $y=1$]

We show by induction on $n\geq 2$ that $F_n(x,1,z) = xzf_n(x+z)$ where $f_n$ is a real-valued function. It is obviously true for $n=2$ with $f_2=1$. Suppose that is true for all $2 \leq k \leq n-1$ with $n \geq 3$. 

\medskip

We denote by $\mathcal{P}_n$ the set of couples $(A,B)$ with $A$ and $B$ two subsets of $\{1,\dots,n\}$ such that $A \cup B = \{1,\dots,n\}$, $A \cap B = \varnothing$, $1 \in A$ and $n \in B$. Let $t\in \mathfrak{C}_n$. Set the vertex 1 to be the root of $t$. Consider cutting the tree $t$ by removing the edge between the vertex $n$ and its parent to end up with two trees $t_1 \in \mathfrak{C}_A$ and $t_2 \in \mathfrak{C}_B$ for some $(A,B) \in \mathcal{P}_n$. Actually given two trees $t_1 \in \mathfrak{C}_A$ and $t_2 \in \mathfrak{C}_B$ with $(A,B) \in \mathcal{P}_n$ there are $|A|$ distinct ways to attach $t_2$ to $t_1$ by joining the vertex $n+1$ of $t_2$ to one of $t_1$'s vertices to obtain a tree $t \in \mathfrak{C}_n$. With this in mind we can decompose the quantity $F_n(x,1,z)$ depending on where $t_2$ is attached to $t_1$:
\begin{equation}
    F_n(x,1,z) = xz + \sum_{\substack{(A,B) \in \mathcal{P}_n \\ |A|>1}} \sum_{\substack{t_1 \in \mathfrak{C}_A \\ t_2 \in \mathfrak{C}_B}} (|A|-2)x^{\deg_1(t_1)}z^{|L_1(t_1)|} + x^{\deg_1(t_1)+1}z^{|L_1(t_1)|} + x^{\deg_1(t_1)}z^{|L_1(t_1)|+1}.
    \label{tree cutting}
\end{equation}
The first term corresponds to the case where $t_1$ has only 1 vertex (i.e. $|A|=1$). If $|A|>1$ then the vertex 1 and the last vertex of the path $L_1(t_1)$ are distinct in $t_1$ and we have to consider three cases: 1) we attach $t_2$ to a vertex of $t_1$ which is neither 1, nor the last vertex of $L_1(t_1)$. 2) We attach $t_2$ to the vertex 1. 3) We attach $t_2$ to the last vertex of $L_1(t_1)$. We then have:
$$
    F_n(x,1,z) = xz + \sum_{a=2}^{n-1} \binom{n-2}{a-1} F_a(x,1,z)(a-2 + x+z).
$$
By induction we conclude that:
$$
    F_n(x,1,z) = xz + xz\sum_{a=2}^{n-1} \binom{n-2}{a-1} f_a(x+z)(a-2 + x+z).
$$
So $F_n(x,1,z)/(xz)$ depends only on $x+z$, thus induction is shown. We then just need to take $z=1$ and use the case $y=z=1$ to conclude.

\end{proof}

To prove the general case of Proposition \ref{prop computation of fn} we will use the particular cases $y=1$ and $z=1$. We will also use the following Abel's binomial Theorem \cite[p.\hspace{0.2em}18]{riordan}.

\begin{prop}
For every integer $n\geq 0$ the following identity holds:
$$
    \sum_{k=0}^{n} \binom{n}{k} x (x-kz)^{k-1} (y+kz)^{n-k} = (x+y)^n.
$$
\end{prop}

Three useful variants can be deduced from this identity.
\begin{cor}
For every integer $n\geq 0$,
\begin{align}
\nonumber
    \text{Variant 1~~~~} & \sum_{k=0}^{n} \binom{n}{k} (x+k)^{k-1} (y-k)^{n-k} = \frac{(x+y)^n}{x} ; \\ \nonumber
    \text{Variant 2~~~~} & \sum_{k=0}^{n} \binom{n}{k} (x+k)^{k} (n-k+y)^{n-k-1} = \frac{(x+y+n)^n}{y} ; \\ \nonumber
    \text{Variant 3~~~~} & \sum_{k=0}^{n} \binom{n}{k} (x+k)^{k-1} (n-k+y)^{n-k-1} = \frac{x+y}{xy}(x+y+n)^{n-1}.
\end{align}
\end{cor}

\begin{proof}
Taking $z=-1$ in Abel's binomial formula gives the first variant. Doing the change of index $k \rightarrow n-k$ and the change of variables $y \rightarrow x+n$ and $x \rightarrow y$ in variant 1 gives variant 2. To get variant 3 we begin by differentiating variant 1 with respect to $y$, so we have:
$$
    \sum_{k=0}^{n} \binom{n}{k} (n-k) (x+k)^{k-1} (n-k+y)^{n-k-1} = n \frac{(x+y+n)^{n-1}}{x}.
$$
Denote by $A$ the left side of variant 3 which we want to compute, then:
$$
    nA - n\sum_{k=1}^{n} \binom{n-1}{k-1} (x+k)^{k-1} (n-k+y)^{n-k-1} = n \frac{(x+y+n)^{n-1}}{x}.
$$
By doing the change of index $k \rightarrow k+1$ in the last sum and using variant 2 for the resulting sum we get:
$$
    nA - n\frac{(x+y+n)^{n-1}}{y} = n\frac{(x+y+n)^{n-1}}{x}.
$$
The expression of $A$ can be deduced from the last display and thus variant 3 is shown.
\end{proof}

\begin{proof}[Proof of Proposition \ref{prop computation of fn} in the general case]

Assume $n \geq 3$. Once again we will use a "tree-cutting" argument but instead of cutting at vertex $n$, we cut at vertex $2$. More precisely, we denote by $\mathcal{Q}_n$ the set of all couples $(A,B)$ with $A$ and $B$ two subsets of $\{1,\dots,n\}$ such that $A \cup B = \{1,\dots,n\}$, $A \cap B = \varnothing$, $1 \in A$ and $2 \in B$. Once again, we decompose the quantity $F_n(x,y,z)$ depending on where $t_2$ is attached to $t_1$: 

$$
    F_n(x,y,z) = \sum_{(A,B) \in \mathcal{Q}_n} \sum_{\substack{t_1 \in \mathfrak{C}_A \\ t_2 \in \mathfrak{C}_B}} \left[ (|A|-1)x^{\deg_1(t_1)}y^{\deg_2(t_2)+1}z^{|L_1(t_1)|} + x^{\deg_1(t_1)+1}y^{\deg_2(t_2)+1}z^{|L_2(t_2)|+1} \right].
$$
The first term appearing after the sums corresponds to attaching $t_2$ to a vertex which is not 1 in $t_1$ and the second term corresponds to attaching $t_2$ to the vertex 1. We then have:

$$
    F_n(x,y,z) = \sum_{a=2}^{n-1} \binom{n-2}{a-1} (a-1)yF_a(x,1,z)F_{n-a}(y) + \sum_{a=1}^{n-1} \binom{n-2}{a-1} xyzF_a(x)F_{n-a}(y,1,z).
$$
Now we can use the cases $y=1$ and $y=z=1$ to replace the occurrences of $F_n$ in the last display. Let's compute the first sum which we call $A_n(x,y,z)$, afterwards we will compute the second one, $B_n(x,y,z)$. 

\begin{align}
    \nonumber
    A_n(x,y,z) & = \sum_{a=2}^{n-1} \binom{n-2}{a-1} (a-1) xy^2z (a-2+x+z)^{a-2}(n-a-1+y)^{n-a-2} \\ \nonumber
    & = (n-2) \sum_{a=0}^{n-3} \binom{n-3}{a} xy^2z (a+x+z)^{a}(n-a-3+y)^{n-a-4} \\ \nonumber
    & = (n-2)xyz(x+y+z+n-3)^{n-3}.
\end{align}
The second equality comes from the fact that $(a-1)\binom{n-2}{a-1} = (n-2)\binom{n-3}{a-2}$. The last equality comes from variant 2 of Abel's binomial formula. Now for $B_n(x,y,z)$ we need to be careful and isolate the case $a=n-1$ because formula (\ref{formula fn}) doesn't apply in the case $(y=1,n=1)$.

\begin{align}
    \nonumber
    B_n(x,y,z) - x^2yz(x+n-2)^{n-3}& = \sum_{a=0}^{n-3} \binom{n-2}{a} x^2y^2z^2 (a+x)^{a-1}(n-a-3+y+z)^{n-a-3} \\ \nonumber
    & = \sum_{a=0}^{n-4} \binom{n-3}{a} x^2y^2z^2 (a+1+x)^{a}(n-a-4+y+z)^{n-a-4} \\ \nonumber
    & ~~~~~ + \sum_{a=0}^{n-3} \binom{n-3}{a} x^2y^2z^2 (a+x)^{a-1}(n-a-3+y+z)^{n-a-3}. \\ \nonumber
    & = \frac{x^2y^2z^2}{y+z-1} \left[ (x+y+z+n-3)^{n-3} - (x+n-2)^{n-3} \right] \\ \nonumber
    & ~~~~~ + xy^2z^2(x+y+z+n-3)^{n-3}.
\end{align}
The second equality comes from the fact that $\binom{n-2}{a} = \binom{n-3}{a-1} + \binom{n-3}{a}$. The last one comes from variants 1 and 2 of Abel's binomial Theorem. The desired formula follows easily.

\end{proof}

\section{A second generating function on Cayley trees and Proof of Theorem \ref{main thm}} \label{study of gn}

The goal now is to compute the exponential generating function of $G_n$:
$$
    \sum_{n \geq 1} \frac{G_{n+1}(x,y,z)}{n!}t^n.
$$
Then by identifying coefficients in formula (\ref{exp gen function of gn}), we will be able to prove Theorem \ref{main thm}. The first step is to establish a symmetry property of $G_n$ with a bijective approach.

\subsection{A symmetry result} \label{a sym result}

Before computing the exponential generating function of $G_n$ we first state a useful symmetry result.

\begin{prop}
For all $n \geq 2$:
$$
G_n(x,y,z) = G_n(y,x,z).
$$
\label{sym of gn}
\end{prop}

\begin{proof}
We will prove it by finding a bijection $\phi$ in $\mathfrak{M}_n$ which exchanges $\mathbb{T}_1$ and $\mathbb{T}_2$ and keeps $\mathbb{M}_1$ unchanged. For $1 \leq k \leq n$ we set $\gamma(k)= 3-k \mod{n}$ so $\gamma$ is a permutation and $\gamma^{-1}=\gamma$. For $(\tau_1,\dots,\tau_{n-1}) \in \mathfrak{M}_n$ we define $\phi(\tau_1,\dots,\tau_{n-1}) = \gamma \circ \tau_1 \circ \dots \circ \tau_{n-1} \circ \gamma$. Notice that for any transposition $\tau = (a~b)$, $\gamma \circ \tau \circ \gamma = (\gamma(a)~\gamma(b))$ hence $\phi(\tau_1,\dots,\tau_{n-1})$ is a product of $n-1$ transpositions. To see that $\phi$ has the expected property, we interpret the action of $\phi$ on the tree $\mathcal{F}$. The tree $\mathcal{F} \circ \phi$ is obtained from $\mathcal{F}$ by relabelling the vertex-labels according to the permutation $\gamma$ and the edge-labels according to the permutation $i \mapsto n-i$. Such an edge-relabelling implies that $\verb&Find&_n$ reads through $\mathcal{E} \circ \phi$ in the exact opposite order than in $\mathcal{E}$. Thus we easily check that $\phi(\tau_1,\dots,\tau_{n-1})$ sends 2 on 3, 3 on 4, $\dots$, $n$ on 1 and 1 on 2 so $\phi(\tau_1,\dots,\tau_{n-1})$ is a minimal factorization of $(1 \dots n)$.
\end{proof}

\subsection{Computation of the exponential generating function of the second generating function} \label{comp EGF of gn}

Before computing the exponential generating function of $G_n$ let's introduce the Lambert $W$ function (see e.g \cite{corless}). It is by definition the solution (in the sense of formal series) of $W(z) e^{W(z)} = z$. Using Lagrange inversion, one can show that $W(z) = \sum_{n \geq 1} \frac{(-n)^{n-1}}{n!} z^n$. Again, using Lagrange inversion, one can also compute:
\begin{equation}
    e^{-rW(-z)} = \left[ \frac{-W(-z)}{z} \right]^r = \sum_{n \geq 0} \frac{r(n+r)^{n-1}}{n!} z^n
    \label{W}
\end{equation}
with the convention that $0 \times 0^{-1} = 1$. These properties of $W$ will be useful when proving Proposition \ref{computation of gn} and Theorem \ref{main thm}.

\begin{prop}
The following identity on formal series holds:
\begin{align}
    \nonumber
    e^t (y-x) \sum_{n \geq 1} \frac{G_{n+1}(x,y,z)}{n!}t^n =  & ~ \frac{xyz(x-1)}{x+z-1} e^{-yW(-t)} - \frac{xyz(y-1)}{y+z-1} e^{-xW(-t)} \\
    & + \frac{xyz^2(y-x)}{(x+z-1)(y+z-1)} e^{-(x+y+z-1)W(-t)}.
\label{exp gen function of gn}
\end{align}
\label{computation of gn}
\end{prop}

\begin{proof}
Fix $n \geq 1$. As in the proof of Proposition \ref{prop computation of fn} in the case $y=1$, we denote by $\mathcal{P}_{n+1}$ the set of couples $(A,B)$ with $A$ and $B$ two subsets of $\{1,\dots,n+1\}$ such that $A \cup B = \{1,\dots,n+1\}$, $A \cap B = \varnothing$, $1 \in A$ and $n+1 \in B$. Then with a similar "tree-cutting" argument that led to (\ref{tree cutting}) we have:

\begin{align}
    \nonumber
    G_{n+1}(x,y,z) = xyz\sum_{t \in \mathfrak{C}_n}y^{\deg_1(t)} + \sum_{\substack{(A,B) \in \mathcal{P}_{n+1} \\ |A|>1}} \sum_{\substack{t_1 \in \mathfrak{C}_A \\ t_2 \in \mathfrak{C}_B}} & \left[ (|A|-2)x^{\deg_1(t_1)}y^{\deg'_2(t_1)}z^{|L_1(t_1)|} \right. \\
    \nonumber
    & + x^{\deg_1(t_1)+1}y^{\deg'_2(t_1)}z^{|L_1(t_1)|} \\ \nonumber
    & \left. + x^{\deg_1(t_1)}y^{\deg_{n+1}(t_2)+1}z^{|L_1(t_1)|+1} \right].
\end{align}
The first term corresponds to the case where $t_1$ has only 1 vertex (i.e. $|A|=1$). If $|A|>1$ then the vertex 1 and the last vertex of the path $L_1(t_1)$ are distinct in $t_1$ and we have to consider three cases: 1) we attach $t_2$ to a vertex of $t_1$ which is neither 1, nor the last vertex of $L_1(t_1)$. 2) We attach $t_2$ to the vertex 1. 3) We attach $t_2$ to the last vertex of $L_1(t_1)$. We then have:
$$
    G_{n+1}(x,y,z) = xyzF_{n}(y) + \sum_{a=2}^{n} \binom{n-1}{a-1} ( (a-2)G_a(x,y,z) + xG_a(x,y,z) + yzF_a(x,1,z)F_{n+1-a}(y) ).
$$
By Proposition \ref{sym of gn}, we get:
$$
    G_{n+1}(x,y,z) = xyzF_{n}(x) + \sum_{a=2}^{n} \binom{n-1}{a-1} \left[ (a-2)G_a(x,y,z) + yG_a(x,y,z) + xzF_a(y,1,z)F_{n+1-a}(x) \right].
$$
Now if we make the difference between the last two equations, we obtain:

\begin{align}
    \nonumber
    0 = ~ & xyz(F_n(y)-F_n(x)) + (x-y)\sum_{a=2}^{n} \binom{n-1}{a-1} G_a(x,y,z) \\ \nonumber
    & + z\sum_{a=2}^{n} \binom{n-1}{a-1} \left[  yF_a(x,1,z)F_{n+1-a}(y)-xF_a(y,1,z)F_{n+1-a}(x) \right]
\end{align}
If $n \geq 2$, by using the third variant of Abel's binomial formula and Proposition \ref{prop computation of fn} we get:
\begin{align}
    \nonumber
    \sum_{a=2}^{n} \binom{n-1}{a-1} F_a(x,1,z)F_{n+1-a}(y) = ~ & xz\frac{x+y+z-1}{x+z-1}(x+y+z+n-2)^{n-2} \\ \nonumber
    & -\frac{xyz(x+n-1)^{n-2}}{y+z-1}
\end{align}
and
\begin{align}
    \nonumber
    \sum_{a=2}^{n} \binom{n-1}{a-1} F_a(y,1,z)F_{n+1-a}(x) = ~ & yz\frac{x+y+z-1}{y+z-1}(x+y+z+n-2)^{n-2} \\ \nonumber
    & -\frac{xyz(y+n-1)^{n-2}}{x+z-1}.
\end{align}
So, finally,
$$
    0 = (x-y)\sum_{a=2}^{n} \binom{n-1}{a-1} G_a(x,y,z) + u_{n-1}(x,y,z),
$$
where
\begin{align}
    \nonumber
    u_{n-1}(x,y,z) =  & ~ \frac{xy^2z(x-1)}{x+z-1} (y+n-1)^{n-2} - \frac{x^2yz(y-1)}{y+z-1} (x+n-1)^{n-2} \\ \nonumber 
    & + \frac{xyz^2(y-x)(x+y+z-1)}{(x+z-1)(y+z-1)} (x+y+z+n-2)^{n-2}. 
\end{align}
Thus, by Pascal's inversion formula,
\begin{equation}
    \label{GnPn}
    (y-x)G_{n+1}(x,y,z) = \sum_{k=0}^n \binom{n}{k} (-1)^{n-k} u_k(x,y,z).
\end{equation}
Now define for $n \geq 0$ the polynomials $P_n$ by
\begin{equation}
    \label{Pn}
    P_n(u) = \sum_{k=0}^n u\binom{n}{k} (-1)^{n-k} (k+u)^{k-1}.
\end{equation}
Their exponential generating function is given by
\begin{align}
    \nonumber
    \sum_{n \geq 0} \frac{P_n(u)}{n!} t^n & = \sum_{k \geq 0} u t^k (k+u)^{k-1} \sum_{n \geq k} \binom{n}{k} \frac{(-1)^{n-k}}{n!} t^{n-k} \\
    \nonumber
    & = u e^{-t} \sum_{k \geq 0} \frac{(k+u)^{k-1}}{k!} t^k \\ \nonumber
    & = e^{-t -uW(-t)}.
\end{align}
This, combined with (\ref{GnPn}), readily gives the desired result.
    
\end{proof}

\subsection{Proof of Theorem \ref{main thm}} \label{proof main thm}

To simplify notation, for $n \geq 2$, set
$$
    p_{i,j}^n = \mathbb{P} \left( \deg_1(T_n)=i, \deg_2'(T_n)=j \right)
$$
so that
$$
    p_{i,j}^n=\frac{1}{n^{n-2}}[x^i y^j]G_n(x,y,1).
$$

\begin{proof}[Proof of Theorem \ref{main thm}]
Fix $i,j \geq 1$. We take the formula of Proposition \ref{computation of gn} with $z=1$ and divide it by $(y-x)$ to get:
\begin{align}
    \nonumber
    e^t \sum_{n \geq 1} \frac{G_{n+1}(x,y,1)}{n!}t^n =  & ~ \frac{xy}{y-x} \left[ e^{-yW(-t)} - e^{-xW(-t)} \right] + \frac{1}{y-x} \left[ xe^{-xW(-t)} - ye^{-yW(-t)} \right] \\ \nonumber
    & + e^{-(x+y)W(-t)}.
\end{align}
We shall identify the coefficient in front of $x^iy^j$ (which is a polynomial in $t$) in this formula. On the left of this equality, the coefficient is:
$$
    e^t \sum_{n \geq 1} p_{i,j}^{n+1}\frac{(n+1)^{n-1}}{n!}t^n.
$$
Using the identity $y^k - x^k = (y-x)(y^{k-1} + y^{k-2}x + \dots + x^{k-1})$ we deduce the coefficient on the right:
$$
    \frac{(-W(-t))^{i+j}}{i!j!} + \frac{(-W(-t))^{i+j-1}}{(i+j-1)!} + \frac{(-W(-t))^{i+j}}{(i+j)!}.
$$
Therefore
$$
    \sum_{n \geq 1} p_{i,j}^{n+1}\frac{(n+1)^{n-1}}{n!}t^n = e^{-t}\frac{(-W(-t))^{i+j}}{i!j!} + e^{-t}\frac{(-W(-t))^{i+j-1}}{(i+j-1)!} + e^{-t}\frac{(-W(-t))^{i+j}}{(i+j)!}.
$$
We now fix $n \geq i+j$ and identify the coefficient associated with $t^n$ in the above formula. Fix $1 \leq \ell \leq n$. By using formula (\ref{W}) we obtain:
$$
    [t^n]e^{-t}(-W(-t))^\ell=\ell\sum_{k=\ell}^n \frac{(-1)^{n-k}}{(n-k)!} \frac{k^{k-1-\ell}}{(k-\ell)!}.
$$
Recall the definition of $P_n$ in (\ref{Pn}). Taking $y=z=1$ in equation (\ref{GnPn}) gives:
$$
    G_{n+1}(x,1,1) = F_{n+1}(x) = P_n(x+1) - P_n(1).
$$
In particular,
$$
    P_n^{(\ell)}(x+1) = F_n^{(\ell)}(x) = x\frac{(n-1)!}{(n-1-\ell)!}(n+x)^{n-1-\ell} + \ell\frac{(n-1)!}{(n-\ell)!}(n+x)^{n-\ell},
$$
with the convention that $1/(-1)!=0$. On the other hand, by definition of $P_n$,
$$
    P_n^{(\ell)}(x) = x\sum_{k=\ell+1}^n (-1)^{n-k} \binom{n}{k} \frac{(k-1)!}{(k-1-\ell)!}(k+x)^{k-1-\ell} + \ell\sum_{k=\ell}^n (-1)^{n-k} \binom{n}{k} \frac{(k-1)!}{(k-\ell)!}(k+x)^{k-\ell}.
$$
We finally obtain:
$$
    \ell\sum_{k=\ell}^n \frac{(-1)^{n-k}}{(n-k)!} \frac{k^{k-1-\ell}}{(k-\ell)!} = \frac{1}{n!}P_n^{(\ell)}(0) = (n-1)^{n-\ell-1}\frac{\ell-1}{(n-\ell)!}.
$$
Formula (\ref{main formula}) then follows.

\end{proof}

\bibliographystyle{alpha}
\bibliography{biblio}

\end{document}